\documentclass[12pt,letterpaper,reqno]{amsart}
\usepackage[letterpaper,margin=1.2in,headheight=15pt]{geometry} 
\usepackage{graphicx,bbm}
\usepackage{amsmath, amssymb, amsfonts, stmaryrd}
\usepackage[utf8]{inputenc}
\usepackage{epstopdf}
\usepackage{tikz-cd} 
\usepackage{setspace}
\usepackage{xcolor}
\usepackage{tcolorbox}
\usepackage[titletoc,title]{appendix}
\usepackage{enumitem}

\usepackage{hyperref}
\definecolor{darkred}{rgb}{0.5,0.15,0.15}
\hypersetup{colorlinks=true,urlcolor=darkred,linkcolor=darkred,citecolor=darkred}

\newcommand{\C}{\mathbb{C}}

\newcommand{\Z}{\mathbb{Z}}
\newcommand{\R}{\mathbb{R}}

\renewcommand{\ell}{X} 

\newcommand{\dif}{\mathrm{d}}

\newcommand{\I}{{\mathrm i}}

\newcommand{\np}{\ensuremath{\mathrm {np}}}

\newcommand{\be}{\begin{eqnarray}}
	\newcommand{\ee}{\end{eqnarray}}
\newcommand{\bea}{\begin{eqnarray}}
	\newcommand{\eea}{\end{eqnarray}}

\newcommand{\ben}{\begin{eqnarray}}
	\newcommand{\een}{\end{eqnarray}}

\providecommand{\Li}{\mathop{\rm Li}\nolimits}

\theoremstyle{plain}
\newtheorem{thm}{Theorem}
\newtheorem{prop}[thm]{Proposition}

\theoremstyle{definition}

\newtheorem*{question*}{Question}

\theoremstyle{remark}
\newtheorem{rem}[thm]{Remark}

\numberwithin{equation}{section}

\address{Fachbereich Mathematik, Universit\"at Hamburg, Bundesstr. 55, 20146, Hamburg}

\title{On the integrable hierarchy for the resolved conifold}

\author{Murad Alim and Arpan Saha}

\begin{document}
	\maketitle
	
	\begin{abstract}
		
		We provide a direct proof of a conjecture of Brini relating the Gromov--Witten theory of the resolved conifold to the Ablowitz--Ladik integrable hierarchy at the level of primaries. In doing so, we use a functional representation of the Ablowitz--Ladik hierarchy as well as a difference equation for the Gromov--Witten potential. In particular, we express certain distinguished solutions of the difference equation in terms of an analytic function which is a specialization of a Tau function put forward by Bridgeland in the study of wall-crossing phenomena of Donaldson--Thomas invariants. 
		
	\end{abstract}

	\section{Introduction}
	The study of integrable structures as the underlying organizing principle of Gromov--Witten (GW) theory has been a great source of insights and interactions between various areas of mathematics such as algebraic and symplectic geometry, the study of integrable systems, and mathematical physics.
	
	The KdV integrable hierarchy for GW theory of a point was put forward in the works of Witten \cite{Witten} and Kontsevich \cite{Kontsevich}, relating the intersection theory on the Deligne--Mumford compactification of the moduli space of curves to $2$d topological gravity. For $\mathbb{P}^1$, the relation of GW theory to the Toda integrable hierarchy was studied in \cite{EHY,EY,EYY,EHX,Pandharipande}. In \cite{Pandharipande}, a conjectured difference equation for the GW potential implied a corresponding difference equation for Hurwitz numbers which was proved in \cite{Okounkov}. 
	
	A more general construction of an integrable hierarchy for a given Gromov--Witten theory was given by Dubrovin and Zhang in \cite{DZ}; see also \cite{Dubrovinrev} for an overview. Yet explicit examples remain sparse; see also \cite{ADKMV} for a relation between topological string theory and integrable hierarchies. For the Calabi--Yau threefold geometry given by the resolved conifold, Brini conjectured in \cite{Brini} that the equivariant descendent Gromov--Witten potential of the resolved conifold with anti-diagonal $\C^\times$ action coincides with the logarithm of a tau function of the Ablowitz--Ladik (AL) hierarchy introduced in \cite{Ablowitz-Ladik} under a suitable identification of the variables. In particular, the conjecture implies that the \emph{primary} GW potential of the resolved conifold (i.\,e.\ without gravitational descendants)  coincides with the restriction of a tau function of the AL hierarchy to the small phase space. 
	
	In \cite{Brini}, Brini provided a proof of the coincidence of the descendent Gromov--Witten potential with the logarithm of an AL tau function to genus $1$, and of the primary Gromov--Witten potential with the restriction of the logarithm of the tau function to the small phase space up to genus $2$. A proof of the conjecture in the stationary sector to all genera was subsequently given by Takasaki  in \cite{Takasaki_2013} using a melting crystal model of Gromov--Witten theory of the resolved conifold. This was essentially a two-step proof: First, the descendent potential was identified with a tau function of the $2$d Toda hierarchy. And then, it was shown that owing to a certain factorization property which the specific tau function in question possessed, it could be interpreted as a tau function of the relativistic Toda hierarchy, a reduction of the $2$d Toda hierarchy which was known to be equivalent to the Ablowitz--Ladik hierarchy.

	In complementary recent developments, Bridgeland put forward in \cite{BridgelandDT} a Riemann--Hilbert problem associated to wall-crossing of Donaldson--Thomas invariants. This was applied in \cite{BridgelandCon} to the resolved conifold, putting forward a Tau function as a solution to the Riemann--Hilbert problem, which in particular provides the Gromov--Witten potential as an asymptotic expansion.
	
	In this work, we give a proof of the identification of the primary Gromov--Witten potential with an AL tau function to all genera. 
	The main ingredients of our proof are two difference equations. One of these is obtained from a certain functional representation of the AL hierarchy due to Vekslerchik \cite{Veks1,Veks2}, while the other was proved by one of us in \cite{alim2020difference} by adapting ideas which appeared in the study of the exact WKB method in \cite{Iwaki}. In particular, by combining this with the results of Bridgeland in \cite{BridgelandCon}, we are able to obtain a closed-form expression for the tau function, thus providing a new, but perhaps expected, link between integrable hierarchies of GW theory and Bridgeland's DT Riemann--Hilbert problem.
	
	The outline of this paper is as follows: In §\ref{sec:GW}, we recall notions from the Gromov--Witten theory of the resolved conifold, introduce the difference equation proved in \cite{alim2020difference},  and construct a solution of it using Bridgeland's Tau function \cite{BridgelandCon}. In §\ref{sec:AL}, we recall Vekslerchik's functional representation of the Ablowitz--Ladik hierarchy. In §\ref{sec:DL}, we show how to take the dispersionless limit of the AL hierarchy in this setting and give a version of Dubrovin's proof of the correspondence between the equivariant genus zero Gromov--Witten potential of the resolved conifold and the dispersionless AL hierarchy adapted to our setting. Finally, in §\ref{sec:Tau}, we bring all the ingredients together to extend a result of Brini.
	
	\subsection*{Acknowledgments}
	We would like to thank Andrea Brini and J\"org Teschner for correspondence and comments on the draft. We also thank the anonymous referee for carefully reading through our submitted draft and offering detailed feedback. This work is supported through the DFG Emmy Noether grant AL 1407/2-1.


	\section{Gromov--Witten potential and difference equation}\label{sec:GW}
	Gromov--Witten (GW) theory of a non-singular algebraic variety $X$ is concerned with the study of integrals over the moduli spaces of maps from Riemann surfaces into $X$.
	Let $X$ be a Calabi--Yau threefold. The GW potential of $X$ is the following formal power series:
	\begin{equation}
		F(\lambda,\boldsymbol t) = \sum_{g\ge 0}  \lambda^{2g-2} F^g(\boldsymbol t)= \sum_{g\ge 0}  \lambda^{2g-2} \sum_{\beta\in H_2(X,\mathbb{Z})}  N^g_{\beta} \,Q^{\beta}\, ,
	\end{equation}
	where $Q^{\beta} := \exp (2\pi \I  \langle\beta,\boldsymbol t\rangle)$ is a formal variable living in a suitable completion of the effective cone in the group ring of $H_2(X,\mathbb{Z})$, $\lambda$ is a formal parameter corresponding to the topological string coupling, and $N^g_{\beta}$ are the GW invariants. 
	
	The GW potential can be written as
	\begin{equation}
		F=F_{\beta=0} + \widetilde{F}\,,
	\end{equation}
	where $F_{\beta=0}$ denotes the contribution from constant maps and $ \widetilde{F}$ the contribution from non-constant maps. The constant-map contributions at genus 0 and 1 are $\boldsymbol t$-dependent and the higher-genus constant-map contributions take the universal form \cite{Faber}
	\begin{equation}
		F_{\beta=0}^g = \frac{(-1)^{g-1}\,\chi(X)\, B_{2g}\, B_{2g-2}}{4g (2g-2)\, (2g-2)!}\,, \quad g\ge2\,,
	\end{equation}
	where $\chi(X)$ is the topological Euler characteristic of $X$ and $B_{j}$ denotes the $j$-th Bernoulli number. 
	
	This note is concerned with an equivariant version of the GW potential of the CY threefold given by the total space of the following rank two bundle over the projective line:
	\begin{equation}
		\mathcal{O}(-1) \oplus \mathcal{O}(-1) \rightarrow \mathbb{P}^1\,.
	\end{equation}
	This corresponds to the resolution of the conifold singularity in $\mathbb{C}^4$ and is known as the resolved conifold. Its \emph{small phase space} $H^\bullet(X;\C)$ is spanned by an identity element $\mathrm{id}$ and the Kähler class $\omega$ of the base $\mathbb{P}^1$. The associated coordinates shall be denoted $x:=\langle \mathrm{id},\boldsymbol t\rangle$ and $t:=\langle \omega,\boldsymbol t\rangle$ respectively. The (non-equivariant) GW potential for this geometry was determined in physics \cite{Gopakumar:1998ii,GV}, and in mathematics \cite{Faber} with the following outcome for the non-constant maps:
	\begin{equation} \label{conpert}
		\widetilde{F}^0= \Li_{3}(Q)\,, \quad \widetilde{F}^g=\frac{(-1)^{g-1}B_{2g}}{2g (2g-2)!}\, \Li_{3-2g} (Q)\,, \quad g\ge1\,,
	\end{equation}
	where $Q:=\exp(2\pi \I t)$ and the polylogarithm is defined by
	\begin{equation}
		\Li_s(z) = \sum_{n=1}^{\infty} \frac{z^n}{n^s}\, ,\quad s\in \mathbb{C}\,.
	\end{equation}
	As an expression for the constant map contribution at genus $1$ we take the following; see e.~g.~\cite{Vaz}:  
	\begin{equation}
		F^1_{\beta=0} = \frac{\pi\I t}{12}+\zeta'(-1)\,,
	\end{equation}
	where $\zeta'(-1)$ denotes the derivative of the Riemann zeta function at $-1$. We note that this constant part of the expression is ambiguous and not defined in Gromov--Witten theory, while the $t$-dependent term has an enumerative meaning; see for instance \cite{threequestions}.
	
	There is an anti-diagonal $\C^\times$ action on $\mathcal{O}(-1) \oplus \mathcal{O}(-1)$, given by fiberwise $\C^\times$ action on the $\mathcal{O}(-1)$ with opposite characters. This $\C^\times$ action is especially interesting as it acts trivially on the canonical bundle, meaning that  $\mathcal{O}(-1) \oplus \mathcal{O}(-1)$ is then \emph{equivariantly} Calabi--Yau. With respect to this action, one can define an equivariant analogue of Gromov--Witten theory where the ordinary cohomology ring of the resolved conifold is replaced with the equivariant cohomology. The equivariant Gromov--Witten invariants of rank $2$ bundles over curves were computed by Bryan and Pandharipande in \cite{BP08}. The theory significantly simplifies in the case of the resolved conifold: It is then just the contribution from genus zero constant maps which is affected. In particular, for the anti-diagonal action, the genus zero Gromov--Witten potential becomes
	\begin{equation}
		F_{\mathrm{ad}}^{ 0} = \frac{(2\pi)^3\I x^2t}{2\kappa^2} + \widetilde F^{ 0}\,,
	\end{equation}
	where  the $\mathrm{ad}$ in the subscript denotes that it is equivariant with respect to the \emph{anti-diagonal} action and $\kappa$ is the equivariant parameter.
	
	The following theorem was proved in \cite{alim2020difference}:
	\begin{thm} \label{diffeq} \cite{alim2020difference} The contribution of the non-constant maps $\widetilde{F}(\lambda,t)$ to the GW potential of the resolved conifold satisfies the following difference equation:
		\begin{equation}\label{eq:diff}
			\widetilde{F}(\lambda,t+\check{\lambda}) - 2 \widetilde{F}(\lambda,t) + \widetilde{F}(\lambda,t-\check{\lambda})= \left(\frac{1}{2\pi }\frac{\partial}{\partial t}\right)^2\, \widetilde{F}^0(t) \,,\quad \check{\lambda}=\frac{\lambda}{2\pi}\,.
		\end{equation}
	\end{thm}
	
	We proceed by obtaining a solution of the difference equation. The latter already appears as a building block of the Tau function determined by Bridgeland in \cite{BridgelandCon} as a solution to a Riemann--Hilbert problem associated to wall-crossing of DT invariants of the resolved conifold. The special functions in \cite{BridgelandCon} involve the multiple sine functions which are defined using the Barnes multiple Gamma functions \cite{Barnes}. For a variable $z\in \mathbb{C}$ and parameters $\omega_1,\ldots,\omega_r \in \mathbb{C}^{\times}$, these are defined by
	\begin{equation}
		\sin_r(z\,|\, \omega_1,\dots,\omega_r):= \Gamma_{r}(z\, |\, \omega_1,\dots,\omega_r) \,\, \Gamma_{r}\bigg(\sum_{i=1}^r \omega_i - z\, |\, \omega_1,\dots,\omega_r\bigg)^{(-1)^r}.
	\end{equation}
	For further definitions, see e.~g.~ \cite{BridgelandCon,Ruijsenaars1} and references therein. We will furthermore need the generalized Bernoulli polynomials, defined by the generating function
	\begin{equation}
		\frac{x^r\, e^{zx}}{ \prod_{i=1}^r (e^{\omega_i x}-1)} = \sum_{n=0}^{\infty} \frac{x^n}{n!} \, B_{r,n}(z\,|\, \omega_1,\,\dots,\omega_r)\,.
	\end{equation}
	Consider now the function $G(z\,|\,\omega_1,\omega_2)$ of \cite[Sec. 4.2]{BridgelandCon}, defined by
	\begin{equation}
		G(z\, | \, \omega_1,\omega_2) := \exp\left(\frac{\pi \I}{6}  B_{3,3}(z+\omega_1\,|\,\omega_1,\omega_1,\omega_2)\right)  \sin_3(z+\omega_1\, |\, \omega_1,\omega_1,\omega_2)\,,
	\end{equation}
	and define for $k\in\Z$ the functions\footnote{The subscript np of the function stands for non-perturbative. In \cite{BridgelandCon}, it is suggested that the Tau function of the DT Riemann--Hilbert problem gives a non-perturbative definition of the GW theory of the resolved conifold. The non-perturbative content of this function is discussed in a companion paper \cite{alim2021intrinsic}.}
	\begin{equation}
		\begin{split}
			F^+_{\np,k}(\lambda,t)&:= \log G(t+k\,|\,\check{\lambda},1)\,, \\ 
			\widehat F^+_{\np,k}(\lambda,t) &:= F^+_{\np,k}(\lambda,t)- F^+_{\np,0}(\lambda,0) - \frac{\zeta(3)}{\lambda^2} - \frac{1}{12} \log \check\lambda + \frac{\pi\I t}{12}+\zeta'(-1)\,,\\
			F^-_{\np,k}(\lambda,t)&:= F^+_{\np,k}(\lambda,-t) - \frac{(2\pi\I)^3}{3!}\frac{B_{1,3}(t\,|\,1)}{ \lambda^2} - \frac{\pi\I}{6}\left(t-\frac{1}{2}\right),\\
			\widehat F^-_{\np,k}(\lambda,t) &:= \widehat F^+_{\np,k}(\lambda,-t) - \frac{(2\pi\I)^3}{3!}\frac{B_{1,3}(t\,|\,1)}{ \lambda^2} + \frac{\pi\I}{12}\,.
		\end{split}
	\end{equation}
	We obtain the following:
	
	\begin{prop}\label{prop:uniqueness}
		The functions $F^\pm_{\np,k}(\lambda,t)$ and $\widehat F^\pm_{\np,k}(\lambda,t)$ are solutions of the difference equation \eqref{eq:diff} i.e. 
		\begin{equation}\label{eq:diff-f}
			f(\lambda,t+\check{\lambda}) - 2f(\lambda,t) + f(\lambda,t-\check{\lambda})= \log(1-Q)\,.
		\end{equation}
		Moreover, $F^\pm_{\np,k}(\lambda,t)$ and $\widehat F^\pm_{\np,k}(\lambda,t)$ have the following asymptotic expansions as $\lambda \rightarrow 0$ subject to $\mathrm{Re}\,\lambda > 0$, with $-k<\pm\mathrm{Re}\,t < -k + 1$ and $\pm\mathrm{Im}\,t > 0$:
		\begin{equation}\label{eq:asymptotic}
			F^\pm_{\np,k} (\lambda,t) \sim \sum_{g=0}^{\infty} \lambda^{2g-2} \widetilde{F}^{g}(t)\,,\quad  \widehat F^\pm_{\np,k} (\lambda,t) \sim  \frac{\widetilde F^0(t)}{\lambda^2}+\sum_{g=2}^{\infty} \lambda^{2g-2} F^{g}(t)\,,
		\end{equation}
		where $F^g(t)$ are the conifold free energies and $\widetilde{F}^g(t)$ are the non-constant-map contributions to them, as defined in \eqref{conpert}. 
	\end{prop}
	\begin{proof}
		Notice that  $\widehat F^\pm_{\np,k}(\lambda,t)$ differs from $F^\pm_{\np,k}(\lambda,t)$ by a function that is a degree $1$ polynomial in $t$. Moreover, the difference equation \eqref{eq:diff-f} is invariant under the replacement $t \mapsto t + k$, while under the replacement $t\mapsto -t$, it becomes
		\begin{equation}
			\begin{split}
				&f(\lambda,-t+\check{\lambda}) - 2f(\lambda,-t) + f(\lambda,-t-\check{\lambda})
				= \log\bigg(1-\frac{1}{Q}\bigg)\\
				&=\log(1-Q) + \frac{(2\pi \I)^3}{3!\lambda^2}\left(B_{1,3}(t+\check{\lambda}\,|\,1) - 2B_{1,3}(t\,|\,1) + B_{1,3}(t-\check{\lambda}\,|\,1)\right).
			\end{split}
		\end{equation}
		Thus, it suffices to prove that $F^+_{\np,0}(\lambda,t)$ is a solution of the difference equation \eqref{eq:diff-f}. 
		
		By \cite[Prop.\ 4.2]{BridgelandCon}, the function $G(t\, |\, \check{\lambda},1)$ is single valued for $\check{\lambda}\notin \mathbb{R}_{<0} $ and satisfies the difference equation
		\begin{equation}\label{diffeqg}
			\frac{G(t + \check{\lambda}\, | \, \check{\lambda},1)}{G(t \, | \, \check{\lambda},1)} = H(t+\check{\lambda}\, | \, \check{\lambda},1)^{-1}\,,
		\end{equation}
		where $H$ is denoted by $F$ in \cite{BridgelandCon} and is given by
		\begin{equation}
			H(z \,| \,\omega_1, \omega_2 ) := \exp\left(-\frac{\pi \I }{2} B_{2,2}(z\,|\,\omega_1,\omega_2)\right) \, \sin_{2}(z\,|\, \omega_1,\omega_2)\,.
		\end{equation}
		This in turn satisfies the difference equation \cite[Prop.\ 4.1]{BridgelandCon}
		\begin{equation}\label{eq:H-dif}
			\frac{H(z + \omega_1 \,|\, \omega_1, \omega_2 )}{H(z \,|\, \omega_1, \omega_2 )}= \frac{1}{1-x_2}\,, \quad x_2=\exp\left(\frac{2\pi \I z}{\omega_2}\right).
		\end{equation}
		From \eqref{diffeqg} we obtain, by taking the logarithm
		\begin{equation}
			\begin{split}
				&\log G(t + \check{\lambda}\, | \, \check{\lambda},1) +  \log G(t - \check{\lambda}\, | \, \check{\lambda},1) -2 \log G(t \, | \, \check{\lambda},1) \\ 
				&= - \left(\log H(t+\check{\lambda}\, | \, \check{\lambda},1) - \log H(t\, | \, \check{\lambda},1) \right) = \log(1-Q)\,, \quad Q=\exp(2\pi \I t)\,.
			\end{split}
		\end{equation}
		For the asymptotic expansions of $F^\pm_{\np,k}$ and $\widehat F^\pm_{\np,k}$, we use \cite[Prop.\ 4.6 and 4.7]{BridgelandCon}, which in particular proves the asymptotic expansion of $G$ as $\omega_1\rightarrow 0$ subject to $\mathrm{Re}\,\omega_1 > 0$, with $0<\mathrm{Re}\,z<\mathrm{Re}\,\omega_2$ and $\mathrm{Im}(z/\omega_2)>0$, to be
		\begin{equation}
			\begin{split}
				\log G(z\, | \, \omega_1,\omega_2) &\sim   \sum_{k= 0}^{\infty} \frac{(k-1)\, B_{k}\, \omega_1^{k-2}}{k!} \, \left( \frac{2\pi \I}{\omega_2}\right)^{k-2} \Li_{3-k} \left( \exp\left(\frac{2\pi \I z}{ \omega_2}\right)\right),\\
				\log G(0\, | \, \omega_1,\omega_2) &\sim  -\zeta(3)\left(\frac{\omega_2}{\omega_1}\right)^2 + \frac{1}{12}\log\bigg(\frac{\omega_2}{\omega_1}\bigg)-\sum_{k= 3}^\infty\frac{B_k\,B_{k-2}}{k(k-2)(k-2)!}\left(\frac{2\pi\I \omega_1}{\omega_2}\right)^{k-2}.
			\end{split}
		\end{equation}
		As a consequence, given that $-k<\pm\mathrm{Re}\,t < -k + 1$ and $\pm\mathrm{Im}\,t > 0$, we have as $\lambda \rightarrow 0$ subject to $\mathrm{Re}\,\lambda > 0$, the asymptotic series
		\begin{equation}
			\begin{split}
				F^\pm_{\np,k}(\lambda,t) &\sim  \frac{1}{\lambda^2}\Li_{3}(Q)+\sum_{g=1}^{\infty} \lambda^{2g-2} \frac{(-1)^{g-1}B_{2g}}{2g (2g-2)!}\, \textrm{Li}_{3-2g}(Q)\,,\\
				\widehat F^\pm_{\np,k} (\lambda,t) &\sim \frac{1}{\lambda^2}\Li_{3}(Q)+\sum_{g=1}^{\infty} \lambda^{2g-2} \frac{(-1)^{g-1}B_{2g}}{2g (2g-2)!}\, \textrm{Li}_{3-2g}(Q)\\
				&\quad +\frac{\pi\I t}{12}+\zeta'(-1)+ \sum_{g=2}^\infty\lambda^{2g-2}\frac{(-1)^{g-1}B_{2g}B_{2g-2}}{2g (2g-2)\, (2g-2)!}\,.
			\end{split}
		\end{equation}
		In the above, we have made use of the fact that $B_{j}=0 $ for $j>1$ odd and that
		\begin{equation}
			\Li_j(e^{2\pi\I t}) + (-1)^j \Li_j (e^{-2\pi\I t}) = \begin{cases}- \frac{(2\pi\I)^j}{j!} B_{1,j}(t\,|\,1)\,, \quad \mbox{when $j\in\Z_{>0}$}\,,\\ 0\,,\qquad\qquad \qquad\quad\,\, \,\, \mbox{when $j\in\Z_{\le 0}$}\,.  \end{cases}
		\end{equation}
		The result to be proved then follows from the fact that the Euler characteristic of the resolved conifold is $2$.

	\end{proof}

	\begin{rem}
		As a corollary, we see that the equivariant Gromov--Witten potential of the resolved conifold with anti-diagonal action admits the following non-perturbative completions:
		\begin{equation}
			F^\pm_{\mathrm{ad},k}(\lambda,\boldsymbol t)=\frac{(2\pi)^3\I x^2t}{2\kappa^2\check{\lambda}^2}+\widehat F^\pm_{\np,k}(\lambda,t)\,.
		\end{equation}
		These aren't the only possibilities, but they are the ones best approximated by the given asymptotic expansion for the Gromov--Witten potential in the sense that they are, for their respective ranges of $t$, the unique holomorphic function in the region $\mathrm{Re}\,\lambda > 0$ such that we have the following bound on the error:
		\begin{equation}
			\left|F^\pm_{\mathrm{ad},k} (\lambda,\boldsymbol t) - \frac{(2\pi)^3\I x^2t}{2\kappa^2\check{\lambda}^2} -  \frac{\widetilde F^0(t)}{\lambda^2}-\sum_{g=2}^{n} \lambda^{2g-2} F^{g}(t)\right| < C(\boldsymbol t)^{n+1}n!|\lambda|^n\,,
		\end{equation}
		where $C(\boldsymbol t)$ is some positive function. In other words, they are obtained by Borel resummation of the  asymptotic series along the positive real line $\R_{>0}$ in the Borel plane. This was explicitly worked out in full detail in \cite{GWtoDT}. In particular, it was determined that the Stokes rays in the Borel plane are at $\pm l_k := \R_{<0}\cdot \pm\I(t+k)$ for all $k\in \Z$ and $\pm l_\infty:=\mathbb R_{<0}\cdot \pm\I$ with the associated Stokes jumps being 
		\begin{equation}
			\begin{split}
				\phi_{\pm l_k} &= \frac{\mathrm{sign}(\mathrm{Im}\,t)}{2\pi \I} \frac{\partial}{\partial \check{\lambda}}\left(\check{\lambda}\Li_2\bigg(\exp\bigg(\pm\frac{2\pi\I(t+k)}{\check{\lambda}}\bigg)\bigg)\right),\\
				\phi_{\pm l_\infty} &= \frac{1}{\pi \I} \sum_{k=1}^\infty\frac{\partial}{\partial \check{\lambda}}\left(\check{\lambda}\Li_2\bigg(\exp\bigg(\pm\frac{2\pi\I k}{\check{\lambda}}\bigg)\bigg)\right) - \frac{\pi \I}{12}\,.
			\end{split}
		\end{equation}
		Notice that for $-k<\mathrm{Re}\,t < -k + 1$ and $\mathrm{Im}\,t > 0$, the positive real line $\R_{>0}$ lies in the sector bounded by $l_{k-1}$ and $l_k$, so the Borel resummation in this sector corresponds to $F^+_{\mathrm{ad},k}(\lambda,t)$. Meanwhile,  for $-k<-\mathrm{Re}\,t < -k + 1$ and $-\mathrm{Im}\,t > 0$,  the positive real line $\R_{>0}$ lies in the sector bounded by $-l_{-k}$ and $-l_{-k+1}$, so the Borel resummation in this sector corresponds to  $F^-_{\mathrm{ad},k}(\lambda,t)$. The  Stokes jumps $\phi_{\pm l_k}$ can then be reproduced using the following identity \cite[p.\ 428]{BridgelandCon}:
		\begin{equation}
			\log G(z+ \omega_2\,|\,\omega_1,\omega_2) - \log G(z\,|\,\omega_1,\omega_2) = \frac{1}{2\pi\I}\frac{\partial}{\partial \omega_1}\left(\omega_1 \Li_2\bigg(\exp\bigg(\frac{2\pi\I z}{\omega_1}\bigg)\bigg)\right).
		\end{equation}
	\end{rem}

	
	
	\section{The Ablowitz--Ladik hierarchy}\label{sec:AL}

	The AL hierarchy \cite{Ablowitz-Ladik} is a set of infinitely many non-linear differential-difference equations in two functions $a,b\colon \Z \rightarrow \C\llbracket z,\widetilde{z} \rrbracket$, where $z$ and $\widetilde{z}$ represent two countably infinite tuples of (complex) flow parameters $z_i$ and $\widetilde{z}_i$ respectively. The simplest non-trivial flow equation in the hierarchy is the AL system:
	\begin{equation}
		\begin{split}
			\I\dot a &= \phantom{+} (\Lambda(a) + \Lambda^{-1}(a))(1-ab) \,,\\
			\I\dot b &= -(\Lambda(b) + \Lambda^{-1}(b))(1-ab)\,.
		\end{split}
	\end{equation}
	Here the dot denotes derivative with respect to a linear combination of the flow parameters in the tuples $z$ and $\widetilde{z}$, while $\Lambda$ denotes the shift operator whose action on any function $u$ on $\Z$ is given by
	\begin{equation}
		(\Lambda (u))(n)= u(n+1)\,.
	\end{equation}
	One way of constructing the full AL hierarchy is by means of three tau functions $\sigma, \rho, \tau\colon \Z \rightarrow \C\llbracket z,\widetilde{z}\rrbracket$ associated to a solution $(a,b)$ of the AL system. These are \emph{defined} by
	\begin{equation}
		a = \frac{\sigma}{\tau}\,, \quad b = \frac{\rho}{\tau}\,, \quad 1 - ab = \frac{\Lambda(\tau)\Lambda^{-1}(\tau)}{\tau^2}\,.
	\end{equation}
	Note that the tau functions are not independent but satisfy
	\begin{equation}
		\Lambda(\tau)\Lambda^{-1}(\tau) = \tau^2 - \sigma \rho\,.
	\end{equation}
	The AL hierarchy is then equivalent to the following system of Hirota bilinear equations (cf.\ Vekslerchik \cite{Veks1,Veks2}):
	\begin{subequations}\label{eq:Hirota}
		\begin{align}
			\tau(z,\widetilde{z}) \tau(z + \I [\zeta], \widetilde{z}) - \rho(z,\widetilde{z}) \sigma(z + \I [\zeta], \widetilde{z}) &= \phantom{+\zeta}\Lambda^{-1}(\tau)(z,\widetilde{z}) \Lambda(\tau)(z+ \I [\zeta], \widetilde{z})\,,\label{eq:Ha}\\
			\tau(z,\widetilde{z}) \sigma(z + \I [\zeta], \widetilde{z}) - \sigma(z, \widetilde{z}) \tau(z + \I [\zeta], \widetilde{z}) &=\phantom{+}\zeta \Lambda^{-1}(\tau)(z,\widetilde{z}) \Lambda(\sigma)(z + \I [\zeta], \widetilde{z})\,,\label{eq:Hb}\\
			\rho(z,\widetilde{z}) \tau(z + \I [\zeta], \widetilde{z}) - \tau(z,\widetilde{z}) \rho(z + \I [\zeta], \widetilde{z}) &=\phantom{+}\zeta \Lambda^{-1}(\rho)(z, \widetilde{z}) \Lambda(\tau)(z + \I [\zeta], \widetilde{z})\,,\label{eq:Hc}\\
			\tau(z, \widetilde{z}) \tau(z,\widetilde{z} + \I [\zeta]) - \rho(z,\widetilde{z}) \sigma(z,\widetilde{z} + \I [\zeta]) &=\phantom{-\zeta} \Lambda(\tau)(z,\widetilde{z}) \Lambda^{-1}(\tau)(z,\widetilde{z} + \I [\zeta])\,,\label{eq:Hd}\\
			\tau(z,\widetilde{z}) \sigma(z,\widetilde{z}+ \I [\zeta]) - \sigma(z,\widetilde{z}) \tau(z,\widetilde{z} + \I [\zeta]) &=-\zeta \Lambda(\tau)(z,\widetilde{z}) \Lambda^{-1}(\sigma)(z,\widetilde{z} + \I [\zeta])\,,\label{eq:He}\\
			\rho(z,\widetilde{z}) \tau(z,\widetilde{z} + \I [\zeta]) - \tau(z,\widetilde{z}) \rho(z,\widetilde{z}+ \I [\zeta]) &=-\zeta \Lambda(\rho)(z,\widetilde{z}) \Lambda^{-1}(\tau)(z,\widetilde{z} + \I [\zeta])\,.\label{eq:Hf}
		\end{align}
	\end{subequations}
	In the above, $[\zeta]$ denotes the tuple
	\begin{equation}
		[\zeta] := \left(\zeta,\frac{\zeta^2}{2},\frac{\zeta^3}{3}, \ldots\right).
	\end{equation}
	The flows of the hierarchy are then obtained by expanding the above set of equations order by order in $\zeta$. In particular, the AL system is obtained by taking the flow along $\partial_{z_1} - \partial_{\widetilde{z}_1}$.
	
	\section{Dispersionless limit}\label{sec:DL}
	
	The dispersionless limit of any integrable hierarchy is a limit in which the contribution of higher-derivative terms (which cause dispersion) is suppressed. This is achieved by scaling the length and times (i.\ e.\ the flow parameters) by a factor $\check{\lambda}^{-1}$ inversely proportional to a dispersion parameter $\lambda$, and then taking the limit $\lambda \rightarrow 0$. Expanding the relevant flow equations as Taylor series in $\lambda$ about $\lambda = 0$ then gives the small-dispersion expansion of the hierarchy.
	
	The most straightforward way to do this in the case of the AL hierarchy is to work at the level of tau functions. That is to say, we introduce tau functions $\sigma_\lambda, \rho_\lambda, \tau_\lambda\colon \C \rightarrow \C\llbracket z,\widetilde{z}\rrbracket$ of a continuous complex variable $x$, parametrised by $\lambda$.  (It is not a coincidence that  the spatial coordinate is denoted with the same symbol as the group ring parameter associated to the identity class in $H^0(X, \C)$; we will later see that they can be identified.) The dependence of the tau functions on $\lambda$ is stipulated to take the following form:
	\begin{equation}\label{eq:lambda-dep}
		\begin{split}
			\log\sigma_\lambda = \frac{\varpi_\lambda}{\check{\lambda}^2} + \frac{ s_\lambda}{\check{\lambda}} +  r_\lambda\,, \quad \log\rho_\lambda = \frac{\varpi_\lambda}{\check{\lambda}^2} - \frac{ s_\lambda}{\check{\lambda}} +  r_\lambda\,, \quad \log \tau_\lambda &= \frac{\varpi_\lambda}{\check{\lambda}^2}\,. 
		\end{split}
	\end{equation}
	Here, $\varpi_\lambda, s_\lambda, r_\lambda$ are analytic in $\lambda$. So, they have well-defined limits $\varpi, s, r$ as $\lambda \rightarrow 0$.
	
	We also introduce the operator $\Lambda_\lambda:= \exp(\check{\lambda}\partial_x)$. 
	The small-dispersion expansion of the AL hierarchy is then given by making the following replacements in \eqref{eq:Hirota}:
	\begin{equation}\label{eq:replace}
		\sigma \mapsto 	\sigma_\lambda\,,\quad \rho \mapsto \rho_\lambda\,, \quad \tau \mapsto \tau_\lambda\,, \quad \Lambda\mapsto \Lambda_\lambda\,, \quad [\zeta] \mapsto \check{\lambda}[\zeta]\,.
	\end{equation}
	In the conjectured correspondence between GW theory and integrable hierarchies, the dispersion parameter $\lambda$ plays the role of the genus expansion parameter. In particular, the genus zero GW theory should correspond to the dispersionless limit of the associated hierarchy. In order to relate the dispersionless limit of the AL hierarchy to the resolved conifold, we will first identify a suitable set of dependent variables and describe the dynamical equations of the hierarchy in terms of them.
	\begin{prop}\label{prop:ALprop}
		The dispersionless limit of the AL hierarchy is given by the following system of equations in the functions $v= \partial_xs$ and $u=-\log(1-e^{2 r})$:
		\begin{subequations}\label{eq:dlAL}
			\begin{align}
				D^\zeta_zv
				&=\phantom{+}\I\,\frac{\partial}{\partial x}\log\left(\frac{1 - \zeta e^{v}+\sqrt{(1+\zeta e^{v})^2 - 4\zeta e^{v}e^{-u}}}{2}\right),\label{eq:dlALa}\\
				D^\zeta_zu &= -\I\,\frac{\partial}{\partial x}\log\left(\frac{1 + \zeta e^{v}+\sqrt{(1+\zeta e^{v})^2 - 4\zeta e^{v}e^{-u}}}{2}\right),
				\label{eq:dlALb}\\
				D^\zeta_{\widetilde{z}}v &= -\I\,\frac{\partial}{\partial x}\log\left(\frac{1 +\zeta e^{-v}+\sqrt{(1-\zeta e^{-v})^2 + 4\zeta e^{-v}e^{-u}}}{2}\right),\label{eq:dlALc}\\
				D^\zeta_{\widetilde{z}}u &=-\I\,\frac{\partial}{\partial x}\log\left(\frac{1 - \zeta e^{-v}+\sqrt{(1-\zeta e^{-v})^2 + 4\zeta e^{-v}e^{-u}}}{2}\right)\label{eq:dlALd},
			\end{align}
		\end{subequations}
		where $D^\zeta_z$ and $D^\zeta_{\widetilde{z}}$ denote the operators
		\begin{equation}
			D^\zeta_z = \sum_{j=1}^\infty\frac{\zeta^j}{j}\frac{\partial}{\partial z_j}\,, \quad D^\zeta_{\widetilde{z}} = \sum_{j=1}^\infty\frac{\zeta^j}{j}\frac{\partial}{\partial \widetilde{z}_j}\,.
		\end{equation}
	\end{prop}
	\begin{proof}
		Making the replacements \eqref{eq:replace} in \eqref{eq:Hirota} and dividing equations \eqref{eq:Ha} and \eqref{eq:Hb} by the respective first terms on the left-hand side, we get
		\begin{subequations}
			\begin{align}
				&1 - \left(\frac{\rho_\lambda}{\tau_\lambda}\right)(z,\widetilde{z})\left(\frac{\sigma_\lambda}{\tau_\lambda}\right)(z + \I\check{\lambda}[\zeta],\widetilde{z})
				\nonumber\\ &\quad=\left(\frac{\Lambda^{-1}_\lambda(\tau_\lambda)}{\tau_\lambda}\right)(z,\widetilde{z})\left(\frac{\Lambda_\lambda(\tau_\lambda)}{\tau_\lambda}\right)(z + \I\check{\lambda}[\zeta],\widetilde{z})\,,\\
				&1 - \left(\frac{(\sigma_\lambda/\tau_\lambda)(z + \I\check{\lambda}[\zeta],\widetilde{z})}{(\sigma_\lambda/\tau_\lambda)(z,\widetilde{z})}\right)^{-1}\nonumber\\
				&\quad = \zeta\left(\frac{\Lambda^{-1}_\lambda(\tau_\lambda)}{\tau_\lambda}\right)(z,\widetilde{z})\left(\frac{\Lambda_\lambda(\tau_\lambda)}{\tau_\lambda}\right)(z + \I\check{\lambda}[\zeta],\widetilde{z})\left(\frac{\Lambda_\lambda(\sigma_\lambda/\tau_\lambda)}{\sigma_\lambda/\tau_\lambda}\right)(z + \I\check{\lambda}[\zeta],\widetilde{z})\,.
			\end{align}
		\end{subequations}
		Substituting \eqref{eq:lambda-dep} into the above then gives us
		\begin{subequations}
			\begin{align}
				&1 - \exp\left(\frac{s_\lambda(z + \I\check{\lambda}[\zeta],\widetilde{z}) - s_\lambda(z,\widetilde{z}))}{\check{\lambda}} \right)\exp\left( r_\lambda(z + \I\check{\lambda}[\zeta],\widetilde{z})+r_\lambda(z,\widetilde{z})\right)
				\nonumber\\ &\quad=\exp\left(\frac{(\Lambda_\lambda^{-1}-1)(\varpi_\lambda)(z ,\widetilde{z}) +(\Lambda_\lambda-1)(\varpi_\lambda)(z + \I\check{\lambda}[\zeta],\widetilde{z})}{\check{\lambda}^2}\right),\label{eq:sub-a}\\
				&1 - \exp\left(-\frac{s_\lambda(z + \I\check{\lambda}[\zeta],\widetilde{z}) - s_\lambda(z,\widetilde{z}))}{\check{\lambda}} \right)\exp\left(-r_\lambda(z + \I\check{\lambda}[\zeta],\widetilde{z})+r_\lambda(z,\widetilde{z}))\right)\nonumber\\
				&\quad = \zeta\exp\left(\frac{(\Lambda_\lambda^{-1}-1)(\varpi_\lambda)(z ,\widetilde{z}) +(\Lambda_\lambda-1)(\varpi_\lambda)(z + \I\check{\lambda}[\zeta],\widetilde{z})}{\check{\lambda}^2}\right)\nonumber\\
				&\quad\quad\times\exp\left((\Lambda_\lambda -1)\left(\frac{s_\lambda(z + \I\check{\lambda}[\zeta],\widetilde{z})}{\check{\lambda}} + r_\lambda(z + \I\check{\lambda}[\zeta],\widetilde{z})\right)\right).\label{eq:sub-b}
			\end{align}
		\end{subequations}
		For convenience, we now introduce the following notation:
		\begin{equation}
			\begin{split}
				X(z,\widetilde{z})&=\frac{s_\lambda(z + \I\check{\lambda}[\zeta],\widetilde{z}) - s_\lambda(z,\widetilde{z}))}{\check{\lambda}}\,,\\
				Y(z,\widetilde{z})&= \frac{ r_\lambda(z + \I\check{\lambda}[\zeta],\widetilde{z})-r_\lambda(z,\widetilde{z})}{\check{\lambda}}\,,\\
				Z(z,\widetilde{z}) &=(\Lambda_\lambda -1)\left(\frac{s_\lambda(z + \I\check{\lambda}[\zeta],\widetilde{z})}{\check{\lambda}} +r_\lambda(z + \I\check{\lambda}[\zeta],\widetilde{z})\right).
			\end{split}
		\end{equation}
		Eliminating $\varpi_\lambda$ from \eqref{eq:sub-a} and \eqref{eq:sub-b} gives us a quadratic equation in $e^{X(z,\widetilde{z})}$:
		\begin{equation}
			\zeta e^{Z(z,\widetilde{z})}e^{\check{\lambda}Y(z,\widetilde{z})}e^{2r_\lambda(z,\widetilde{z})}e^{2X(z,\widetilde{z})} + \left(1-\zeta e^{Z(z,\widetilde{z})}\right)e^{X(z,\widetilde{z})} - e^{-\check{\lambda}Y(z,\widetilde{z})}=0\,.
		\end{equation}
		This has two solutions but only one such that $X(z,\widetilde{z})$ has a well-defined limit as $\zeta \rightarrow 0$:
		\begin{equation}
			X(z,\widetilde{z}) = \log\left(\frac{-1 + \zeta e^{Z(z,\widetilde{z})}+\sqrt{(1-\zeta e^{Z(z,\widetilde{z})})^2 + 4\zeta e^{Z(z,\widetilde{z})}e^{2r_\lambda(z,\widetilde{z})}}}{2\zeta e^{Z(z,\widetilde{z})}e^{\check{\lambda}Y(z,\widetilde{z})}e^{2r_\lambda(z,\widetilde{z})}}\right).
		\end{equation}
		In the dispersionless limit $\lambda \rightarrow 0$, this becomes
		\begin{equation}\label{eq:Dzs}
			\I D^\zeta_zs = \log\left(\frac{-1 + \zeta e^{v}+\sqrt{(1-\zeta e^{v})^2 + 4\zeta e^{v}e^{2r}}}{2\zeta e^{v}e^{2r}}\right).
		\end{equation}
		Differentiating with respect to $x$ on both sides and substituting $e^{2r}=1-e^{-u}$ then gives us
		\begin{equation}
			D^\zeta_zv = -\I\,\frac{\partial}{\partial x}\log\left(\frac{-1 + \zeta e^{v}+\sqrt{(1+\zeta e^{v})^2 - 4\zeta e^{v}e^{-u}}}{2\zeta e^{v}(1-e^{-u})}\right).
		\end{equation}
		Equation \eqref{eq:dlALa} then follows from the identity
		\begin{equation}
			\frac{-1 + \zeta e^{v}+\sqrt{(1+\zeta e^{v})^2 - 4\zeta e^{v}e^{-u}}}{2\zeta e^{v}(1-e^{-u})} = \frac{2}{1 - \zeta e^{v}+\sqrt{(1+\zeta e^{v})^2 - 4\zeta e^{v}e^{-u}}}\,.
		\end{equation}
		We will now use what we have established so far to prove \eqref{eq:dlALb}. First, we note that \eqref{eq:sub-a} in the limit $\lambda \rightarrow 0$ becomes
		\begin{equation}\label{eq:varpi1}
			1- \exp(\I D^\zeta_z s)e^{2r} = \exp\left(\frac{\partial^2\varpi}{\partial x^2} + \I\frac{\partial D^\zeta_z\varpi}{\partial x}\right).
		\end{equation}
		Setting $\zeta = 0$ then gives us
		\begin{equation}\label{eq:varpi2}
			u=-\log(1-e^{2r})=-\frac{\partial^2\varpi}{\partial x^2}\,.
		\end{equation}
		Substituting this back into \eqref{eq:varpi1}, we obtain
		\begin{equation}
			\frac{\partial D^\zeta_z\varpi}{\partial x} = -\I\log\left(\frac{1- \exp(\I D^\zeta_z s)e^{2r}}{e^{-u}}\right).
		\end{equation}
		Differentiating with respect to $x$ on both sides and substituting \eqref{eq:Dzs} and $e^{2r}=1-e^{-u}$ then gives
		\begin{equation}
			D^\zeta_zu = \I\,\frac{\partial}{\partial x}\log\left(\frac{1 + \zeta e^{v}-\sqrt{(1+\zeta e^{v})^2 - 4\zeta e^{v}e^{-u}}}{2\zeta e^{v}e^{-u}}\right).
		\end{equation}
		Equation \eqref{eq:dlALb} then follows from the identity
		\begin{equation}
			\frac{1 + \zeta e^{v}-\sqrt{(1+\zeta e^{v})^2 - 4\zeta e^{v}e^{-u}}}{2\zeta e^{v}e^{-u}}=\frac{2}{1 + \zeta e^{v}+\sqrt{(1+\zeta e^{v})^2 - 4\zeta e^{v}e^{-u}}}\,.
		\end{equation}
		A similar argument involving \eqref{eq:Hd} and \eqref{eq:He} leads to \eqref{eq:dlALc} and \eqref{eq:dlALd}.
	\end{proof}
	
	\begin{rem}\label{rem:xdif}
		By setting $\zeta = 0$ in \eqref{eq:sub-a} and taking the logarithm of both sides, we obtain a difference equation involving shifts in $x$:
		\begin{equation}
			\log(1-\exp(2r_\lambda)) = \frac{\Lambda_\lambda(\varpi_\lambda)-2\varpi_\lambda + \Lambda_\lambda^{-1}(\varpi_\lambda)}{\check{\lambda}^2}\,.
		\end{equation}
		Interestingly, this becomes  the difference equation \eqref{eq:diff-f} under the following replacements:
		\begin{equation}
			x\mapsto t\,, \quad \exp(2r_\lambda) \mapsto Q\,.
		\end{equation}
	\end{rem}
	
	\begin{rem}
		We would like to thank the referee for pointing out to us that the expressions appearing in Prop. \ref{prop:ALprop}, especially the r.~h.~s.~of \eqref{eq:Dzs} coincide with one of the branches of the superpotential which defines the Hori--Iqbal--Vafa \cite{HV,HIV} mirror curve of the resolved conifold. While we do not have an explanation for this observation, we think that it may provide a link between the integrable hierarchies attached to topological string theories in \cite{ADKMV} which have an interpretation in terms of free fermions and open topological string theory and the integrable hierarchy associated to Gromov--Witten theory which corresponds to a closed topological string theory.
		
	\end{rem}
	
	\begin{rem}
		By applying the operator $(\zeta \partial_\zeta)^2$ to \eqref{eq:dlAL} we get the following alternative form for the dispersionless AL hierarchy:
		\begin{subequations}
			\begin{align}
				\Delta^\zeta_zv
				&=-\I\,\frac{\partial}{\partial x}\left(\frac{(1 - \zeta e^{v})\zeta e^{v}e^{-u}}{\sqrt{((1+\zeta e^{v})^2 - 4\zeta e^{v}e^{-u})^3}}\right)\nonumber\\
				&=\phantom{+}\I\,\frac{\partial^2}{\partial x \partial u}\left(\frac{1 - \zeta e^{v}-\sqrt{(1+\zeta e^{v})^2 - 4\zeta e^{v}e^{-u}}}{2\sqrt{(1+\zeta e^{v})^2 - 4\zeta e^{v}e^{-u}}}\right),\label{eq:dlALa2}\\
				\Delta^\zeta_zu &= \phantom{+}\I\,\frac{\partial}{\partial x}\left(\frac{(1-e^u)(1 + \zeta e^{v})\zeta e^{v}e^{-u}}{\sqrt{((1+\zeta e^{v})^2 - 4\zeta e^{v}e^{-u})^3}}\right)\nonumber\\
				&=\phantom{+}\I\,\frac{\partial^2}{\partial x \partial v}\left(\frac{1 - \zeta e^{v}-\sqrt{(1+\zeta e^{v})^2 - 4\zeta e^{v}e^{-u}}}{2\sqrt{(1+\zeta e^{v})^2 - 4\zeta e^{v}e^{-u}}}\right),
				\label{eq:dlALb2}\\
				\Delta^\zeta_{\widetilde{z}}v &=-\I\,\frac{\partial}{\partial x}\left(\frac{(1 + \zeta e^{-v})\zeta e^{-v}e^{-u}}{\sqrt{((1-\zeta e^{-v})^2 + 4\zeta e^{-v}e^{-u})^3}}\right)\nonumber\\
				&=-\I\,\frac{\partial^2}{\partial x \partial u}\left(\frac{1 + \zeta e^{-v}-\sqrt{(1-\zeta e^{-v})^2 + 4\zeta e^{-v}e^{-u}}}{2\sqrt{(1-\zeta e^{-v})^2 + 4\zeta e^{-v}e^{-u}}}\right),\label{eq:dlALc2}\\
				\Delta^\zeta_{\widetilde{z}}u &=-\I\,\frac{\partial}{\partial x}\left(\frac{(1-e^u)(1 - \zeta e^{-v})\zeta e^{-v}e^{-u}}{\sqrt{((1-\zeta e^{-v})^2 + 4\zeta e^{-v}e^{-u})^3}}\right)\nonumber\\
				&=-\I\,\frac{\partial^2}{\partial x \partial v}\left(\frac{1 + \zeta e^{-v}-\sqrt{(1-\zeta e^{-v})^2 + 4\zeta e^{-v}e^{-u}}}{2\sqrt{(1-\zeta e^{-v})^2 + 4\zeta e^{-v}e^{-u}}}\right)\label{eq:dlALd2},
			\end{align}
		\end{subequations}
		where $\Delta^\zeta_z$ and $\Delta^\zeta_{\widetilde{z}}$ denote the operators
		\begin{equation}
			\Delta^\zeta_z = \sum_{j=1}^\infty j\zeta^j\frac{\partial}{\partial z_j}\,, \quad \Delta^\zeta_{\widetilde{z}} = \sum_{j=1}^\infty j\zeta^j\frac{\partial}{\partial \widetilde{z}_j}\,.
		\end{equation}
		In particular, this is a dispersionless \emph{Hamiltonian} hierarchy with generating functions of Hamiltonian densities
		\begin{equation}\label{eq:hamd}
			\begin{split}
				h(\zeta) &= \phantom{+}\I\left(\frac{1 - \zeta e^{v}-\sqrt{(1+\zeta e^{v})^2 - 4\zeta e^{v}e^{-u}}}{\sqrt{(1+\zeta e^{v})^2 - 4\zeta e^{v}e^{-u}}}\right),\\
				\widetilde h(\zeta) &=-\I\left(\frac{1 + \zeta e^{-v}-\sqrt{(1-\zeta e^{-v})^2 + 4\zeta e^{-v}e^{-u}}}{\sqrt{(1-\zeta e^{-v})^2 + 4\zeta e^{-v}e^{-u}}}\right),
			\end{split}
		\end{equation}
		and the following Poisson brackets:
		\begin{equation}\label{eq:Poisson}
			\{u(x),u(y)\} = \{v(x),v(y)\} = 0\,, \quad \{u(x),v(y)\} = \{v(x),u(y)\} = \delta'(x-y)\,.
		\end{equation}
	\end{rem}
	
	There are a number of interesting aspects of the Hamiltonian formulation of the dispersionless limit of the AL hierarchy. Firstly, the Hamiltonian densities depend only on $u$ and $v$ explicitly and not on their spatial derivatives. Secondly, the coefficients of $\delta'(x)$ in the Poisson brackets form a constant symmetric matrix. A Hamiltonian hierarchy that has this structure (possibly after a suitable change of dependent variables) is said to be of \emph{hydrodynamic type}. Such hierarchies were first defined and studied by Dubrovin and Novikov \cite{DN84}. 
	
	A key fact concerning Hamiltonian hierarchies of hydrodynamic type is that there is a natural such hierarchy called the \emph{principal hierarchy} associated to any (almost) Frobenius manifold, which conversely may be viewed as a differential-geometric structure on the initial data of such a hierarchy. The small phase space $H^\bullet(X,\C)$ of a variety with vanishing odd cohomologies happens to carry the structure of a(n almost) Frobenius manifold, so this gives a rather general construction of a dispersionless integrable hierarchy associated to the (equivariant) genus zero GW theory of any such variety. It was noted by Dubrovin and Brini in \cite{Dubrovin,Brini} that the dispersionless limit of the AL hierarchy is contained within the principal hierarchy associated to the equivariant GW theory of the resolved conifold with anti-diagonal action (hereafter referred to as simply the principal hierarchy of the resolved conifold). We will now present a proof of this result in our notation to ensure that the dependent variables $u$ and $v$ that we have introduced are indeed the variables of the principal hierarchy identified by Dubrovin.

	\begin{prop}\label{prop:principal}
		The dispersionless limit of AL hierarchy with the initial conditions $v=2\pi\I x/\kappa$ and $u=-2\pi\I t$
		is contained in the principal hierarchy of the resolved conifold with the following identification on the small phase space: 
		\begin{equation}\label{eq:id}
			\left(\frac{2\pi}{\kappa}\right)^2\varpi|_{(z,\widetilde z)=(0,0)} = \left.\left(\frac{uv^2}{2} + \widetilde F^0\right)\right|_{(z,\widetilde z)=(0,0)}.
		\end{equation}
	\end{prop}
	
	\begin{proof}
		The principal hierarchy associated to a(n almost) Frobenius manifold $V$ is defined on the formal loop space $L(S^1,V)$, whose elements are maps from $S^1=\R/\Z$ to $V$. In our case, $V$ is the small phase space $H^\bullet(X,\mathbb{Z})$, while $w^1:=v$ and $w^2:=u$ can be interpreted as maps from $S^1$ to $V$. (Note that the map $w^1$ is distinguished.) Meanwhile, $x$ is interpreted as a coordinate on the universal cover $\R$ of $S^1$.
		
		A potential $\Phi$ on $V$ encodes a metric $\eta$ and an associative product $\star$ with identity $\partial_{w^1}$ on the tangent bundle $TV$:
		\begin{equation}
			\begin{split}
				\eta_{\mu\nu}:=\eta\left(\frac{\partial}{\partial w^\mu}\,,\frac{\partial}{\partial w^\nu}\right) &= \frac{\partial^3 \Phi}{\partial w^1\partial w^\mu\partial w^\nu}\,, \\ \eta\left(\frac{\partial}{\partial w^\kappa}\,,\frac{\partial}{\partial w^\mu}\star\frac{\partial}{\partial w^\nu}\right) &= \frac{\partial^3\Phi}{\partial w^\kappa\partial w^\mu\partial w^\nu}\,.
			\end{split}
		\end{equation}
		We also have a Poisson bracket on $L(S^1,V)$ in terms of the the inverse metric $\eta^{\mu\nu}$:
		\begin{equation}\label{eq:Poisson2}
			\{w^\mu(x), w^\nu(y)\} = \eta^{\mu\nu}\delta'(x-y) - \eta^{\mu\kappa}\Gamma^\nu_{\kappa\gamma}\frac{\partial w^\gamma}{\partial x}\delta(x-y)\,,
		\end{equation}
		where $\Gamma^\nu_{\kappa\gamma}$ are the Christoffel symbols of the metric $\eta$ and repeated indices are summed over. Comparing this with \eqref{eq:Poisson}, we immediately obtain
		\begin{equation}
			\eta_{11} = \eta_{22} = 0, \quad \eta_{12} = \eta_{21}=1\,.
		\end{equation}
		Thus, the most general form that the potential $\Phi$ can take is
		\begin{equation}
			\Phi = \frac{uv^2}{2} + f(u)\,.
		\end{equation}
		A consequence of the associativity of the product $\star$ is the existence of a pencil of flat torsion-free connections $\nabla^\xi$ parametrised by a parameter $\xi \in \C$:
		\begin{equation}
			\nabla^\xi_{\partial/\partial w^\mu}\frac{\partial}{\partial w^\nu} = \xi\,\frac{\partial}{\partial w^\mu}\star \frac{\partial}{\partial w^\nu}\,.
		\end{equation}
		As these connections are flat, they admit a local basis of parallel $1$-forms $\alpha_1(\xi)$ and $\alpha_2(\xi)$ analytic in $\xi$. And as these connections are torsion-free, the anti-symmetric parts of $\nabla^\xi\alpha_1$ and $\nabla^\xi \alpha_2$ are $\mathrm d\alpha_1$ and $\mathrm d\alpha_2$ (where $\mathrm d$ is the exterior derivative on $V$). Thus, the $1$-forms $\alpha_1(\xi)$ and $\alpha_2(\xi)$ are closed, and so by Poincaré's lemma, we can locally find functions $g_1(\xi)$ and $g_2(\xi)$ analytic in $\xi$ such that we have
		\begin{equation}
			\alpha_1 = \mathrm d g_1\,, \quad  \alpha_2 = \mathrm d g_2\,.
		\end{equation}
		More explicitly, the functions $g_1$ and $g_2$ are solutions to the following system of equations:
		\begin{equation}\label{eq:WDVV}
			\frac{\partial^2g}{\partial v^2} = \xi\, \frac{\partial g}{\partial v}\,, \quad \frac{\partial^2g}{\partial v\,\partial u} = \xi\, \frac{\partial g}{\partial u}\,, \quad \frac{\partial^2 g}{\partial u^2}= \xi f'''(u)\, \frac{\partial g}{\partial v}\,.
		\end{equation}
		The functions $g_1(\xi)$ and $g_2(\xi)$ are generating functions of Hamiltonian densities, which generate the principal hierarchy via the Poisson bracket \eqref{eq:Poisson2}. The generating functions $h(\zeta)$ and $\widetilde h(\zeta)$ are related to these in a non-trivial way (see Remark \ref{rem:Mellin}). We can get around the need to obtain an explicit relationship between them by eliminating $\xi$ from the first and third equations in \eqref{eq:WDVV}.  This gives us the following equation that by virtue of its linearity is valid for any generating function of Hamiltonian densities:
		\begin{equation}\label{eq:density-constraint}
			\frac{\partial^2g}{\partial u^2} = f'''(u)\,\frac{\partial^2g}{\partial v^2}\,.
		\end{equation}
		Substituting $g=h$ and $g=\widetilde h$ as in \eqref{eq:hamd} into the above, we obtain
		\begin{equation}
			f'''(u) = \frac{1}{1-e^{u}}=-\Li_0(e^{-u})\,.
		\end{equation}
		This is solved by $f(u) = \widetilde F^0(t)=\Li_3(e^{2\pi \I t})$ for $u=-2\pi\I t$. So the dispersionless AL hierarchy is contained in the principal hierarchy of the resolved conifold. The identification \eqref{eq:id} meanwhile follows from \eqref{eq:varpi2} and the initial condition $v=2\pi\I x/\kappa$.
	\end{proof}
	
	The exponential $e^\Phi$ of the potential $\Phi$ of the (almost) Frobenius manifold will be henceforth called the \emph{topological tau function} of the principal hierarchy.
	
	\begin{rem}\label{rem:Mellin}
		The first two equations of \eqref{eq:WDVV} imply that any solution $g$ necessarily has the form
		\begin{equation}
			g = A(\xi,u)e^{\xi v} + B(\xi)\,,
		\end{equation}
		where $A$ is a function independent of $v$ while $B$ is a function independent of $v$ and $u$. Since we are only interested in the flows of the Hamiltonians, we can without loss of generality set $B$ to be identically zero. Conversely, any function $g = A(\xi,u)e^{\xi v}$ which also satisfies \eqref{eq:density-constraint} is a solution of \eqref{eq:WDVV}. Now consider the following Mellin transforms, which are well-defined for $u>0$:
		\begin{subequations}
			\begin{align}
				\int_0^{\infty} h \zeta^{ -1-\xi}\dif\zeta &= \phantom{+}\I e^{\xi v}\int_0^{\infty} \left(\frac{1 -\zeta -\sqrt{(1+\zeta )^2 - 4\zeta e^{-u}}}{2\sqrt{(1+\zeta )^2 - 4\zeta e^{-u}}}\right)\zeta^{-1-\xi}\dif\zeta\,,\label{eq:Mela}\\
				\int_0^{\infty} \widetilde h \zeta^{-1+\xi}\dif\zeta &=  -\I e^{\xi v}\int_0^{\infty}\left(\frac{1 +\zeta -\sqrt{(1-\zeta )^2 + 4\zeta e^{-u}}}{2\sqrt{(1-\zeta )^2 + 4\zeta e^{-u}}}\right)\zeta^{-1+\xi}\dif\zeta\,.\label{eq:Melb}
			\end{align}
		\end{subequations}
		These are of the form $A(\xi,u)e^{\xi v}$. Moreover, being integral transforms of $h$ and $\widetilde h$, the resulting functions satisfy \eqref{eq:density-constraint} with $f(u) = \Li_3(e^{-u})$ as well. Hence, it must be a solution $g$ of \eqref{eq:WDVV}. In other words, it is some linear combination of the Hamiltonian density generating functions $g_1$ and $g_2$ of the principal hierarchy of the resolved conifold, with coefficients depending only on $\xi$.

		In fact, we can see this more explicitly. Upon carrying out the following substitution (assuming $u>0$):
		\begin{equation}
			\zeta = \frac{1-y}{y(1-e^{-u}y)}\,, \quad y =\frac{1+\zeta - \sqrt{(1+\zeta)^2 - 4\zeta e^{-u}}}{2\zeta e^{-u}}\,,
		\end{equation}
		we see that the integral in \eqref{eq:Mela} becomes
		\begin{equation}
			\begin{split}
				&\I e^{\xi v}\int_0^{\infty} \left(\frac{1 -\zeta -\sqrt{(1+\zeta )^2 - 4\zeta e^{-u}}}{2\sqrt{(1+\zeta )^2 - 4\zeta e^{-u}}}\right)\zeta^{-1-\xi}\dif\zeta\\
				&= \I e^{\xi v}\int_0^{1} \frac{y^{\xi-1}(1-y)^{-\xi}}{(1-e^{-u} y)^{-\xi}}\mathrm dy\\
				&=\frac{\I \pi  e^{\xi v}}{\sin(\pi\xi)} {}_2 F_1(-\xi,\xi;1;e^{-u})\,, \quad \mbox{when $0 < \mathrm{Re}\,\xi < 1$.}
			\end{split}
		\end{equation}
		Here, ${}_2 F_1$ denotes the hypergeometric function.
		
		Meanwhile, upon carrying out the following substitution (again assuming $u>0$):
		\begin{equation}
			\zeta = \frac{y(1-(1-e^{-u})y)}{1-y}\,, \quad y =\frac{2\zeta}{ 1+\zeta + \sqrt{(1-\zeta)^2 + 4\zeta e^{-u}}}\,,
		\end{equation}
		we see that the integral in \eqref{eq:Melb} becomes
		\begin{equation}
			\begin{split}
				&-\I e^{\xi v}\int_0^{\infty}\left(\frac{1 +\zeta -\sqrt{(1-\zeta )^2 + 4\zeta e^{-u}}}{2\sqrt{(1-\zeta )^2 + 4\zeta e^{-u}}}\right)\zeta^{-1+\xi}\dif\zeta\\
				&= -\I e^{\xi v}(1-e^{-u})\int_0^{1} \frac{y^{\xi}(1-y)^{-\xi}}{(1-(1-e^{-u}) y)^{-\xi+1}}\mathrm dy\\
				&=-\frac{\I \pi \xi e^{\xi v}}{\sin(\pi\xi)}(1-e^{-u}) {}_2 F_1(-\xi+1,\xi+1;2;1-e^{-u})\,, \quad \mbox{when $-1 < \mathrm{Re}\,\xi < 1$.}
			\end{split}
		\end{equation}
		These can be directly verified to be linearly independent solutions of \eqref{eq:WDVV} with $f(u) = \Li_3(e^{-u})$; see, for instance, \cite[Prop.\ 2.5]{Brini}.
	\end{rem}

	\section{Identification of the Tau function}\label{sec:Tau}
	In this section, we will extend the genus zero identification of the dispersionless AL tau function with the topological tau function of the principal hierarchy of the resolved conifold to an all-genera small-phase-space identification of the tau function of the small-dispersion expansion of the AL hierarchy with a suitably normalised version of Bridgeland's Tau function. The genus $2$ truncation of this identification was essentially the content of Theorem 1.4 in \cite{Brini} and it was proved by analyzing dispersive perturbations of the principal hierarchy. This is a computationally challenging task, however the control over expansion in the dispersive parameter offered by the  functional representation of the AL hierarchies will allow us to sidestep these difficulties.
	
	Here is the gist of our approach. The solution of Prop.~\ref{prop:uniqueness} is used to define initial conditions for the tau function of the AL hierarchy while the other difference equation in Remark \ref{rem:xdif} is used to show that these initial conditions reduce to those of the principal hierarchy of the resolved conifold in the dispersionless limit.  
	
	
	We now prove our main theorem, which extends Theorem 1.4 in \cite{Brini} to all genera. 
	
	\begin{thm}\label{thm:main}
		There exists a solution to the small-dispersion expansion of the AL hierarchy satisfying the following conditions:
		\begin{enumerate}
			\item \label{item1} The tau function $\tau_\lambda$ of the solution reduces at leading order in $\lambda$ to the topological tau function of the principal hierarchy of the resolved conifold as $\lambda \rightarrow 0$,
			\item \label{item2} The restriction of the tau function $\tau_\lambda$ of the solution to the small phase space coincides with the exponential of the equivariant GW potential of the resolved conifold with anti-diagonal action, namely
			\begin{equation}\label{eq:expF}
				\exp(F_{\mathrm{ad}}(\lambda,t))=\exp\bigg( \frac{(2\pi)^3\I x^2t}{2\kappa^2\check{\lambda}^2}+\widehat F^+_{\np,0}(\lambda,t)\bigg)\,,
			\end{equation}
			to all orders in $\lambda$.
		\end{enumerate}
	\end{thm}
	\begin{proof}
		
		We impose the following initial conditions on $v_\lambda := \check{\lambda}^{-1}(\Lambda_\lambda -1)s_\lambda$ and $\varpi_\lambda$: 
		\begin{equation}
			v_\lambda|_{(z,\widetilde z)=(0,0)}=\frac{2\pi \I x}{\kappa}\,,\quad \left(\frac{2\pi}{\kappa}\right)^2\varpi_\lambda|_{(z,\widetilde z)=(0,0)} =  \frac{(2\pi)^3\I x^2t}{2\kappa^2} + \check{\lambda}^2\widehat F^+_{\np,0}(\lambda,t)\,.
		\end{equation}
		This entails condition \ref{item2}. Moreover, the difference equation in Remark \ref{rem:xdif} implies the following initial condition for $u_\lambda:=-\log(1-\exp(2r_\lambda))$:
		\begin{equation}
			u_\lambda|_{(z,\widetilde z)=(0,0)}=-2\pi\I t\,.
		\end{equation}
		The initial conditions for $u_\lambda$ and $v_\lambda$ reduce to those for $u$ and $v$ in the principal hierarchy of the resolved conifold in the limit $\lambda \rightarrow 0$. Since the full solution to the principal hierarchy is fixed by the initial conditions, condition \ref{item1} holds as well. 
	\end{proof}
	
	\begin{rem}
		Notice that the proof of Theorem \ref{thm:main} still holds if we replace $\widehat F^+_{\np,0}(\lambda,t)$ by any other solution of the difference equation \eqref{eq:diff}. However, our eventual goal in the future is to make an identification of the descendent potential with the Ablowitz--Ladik tau function at all times. This identification is expected to be sensitive to the precise choice of initial conditions. For this reason,  Theorem \ref{thm:main} has been formulated in terms of the distinguished solution $\widehat F^+_{\np,0}(\lambda,t)$ obtained from Bridgeland's Tau function which corresponds to Borel resummation in a chamber preserved under S-duality; see \cite[Sec.\ 6.1]{GWtoDT}.
	\end{rem}

	\newcommand{\etalchar}[1]{$^{#1}$}

\end{document}